\documentclass[reqno]{amsart}
\usepackage{enumitem}
\usepackage[all]{xy} 
\usepackage{multirow} 
\usepackage{tabu}
\usepackage{amssymb}
\usepackage{leftindex}
\usepackage{bm}
\usepackage{array,longtable}
\usepackage[T1]{fontenc}
\usepackage{ltablex}
\usepackage{lmodern}
\usepackage{float}
\usepackage{tikz-cd}
\usepackage[hidelinks]{hyperref}
\usepackage{tabularx,booktabs,makecell,array}

\newtheorem{theorem}{Theorem}[section]
\newtheorem{lemma}[theorem]{Lemma}

\newtheorem{prop}[theorem]{Proposition}

\newtheorem{corollary}[theorem]{Corollary}

\theoremstyle{definition}
\newtheorem{definition}[theorem]{Definition}

\newtheorem{remark}[theorem]{Remark}
\numberwithin{equation}{section}

\newcommand{\Br}{\operatorname{Br}}

\title{Torsors of Generalized del Pezzo Tori and  Brauer Groups}

\author[Pankaj Singh]{Pankaj Singh}

\address{Department of Mathematics, University of South Carolina, 
Columbia, SC 29208}
\email{pksingh@email.sc.edu}
\date{\today}


\begin{document}

\begin{abstract}
Using cohomological Mackey functors, we give an explicit Brauer group
classification of torsors under tori of generalized del Pezzo
varieties. This identifies the image of Lamarche's
Brauer invariants and extends the Brauer-theoretic description in Blunk's
classification of degree-six del Pezzo surfaces.

We also determine which finite field extensions can appear in other
Brauer group presentations of these torsors. The same tori act on 
 Losev-Manin spaces, so the computation applies to these toric varieties as well.
\end{abstract}

\maketitle


\section{Introduction}
\label{sec:intro}

Let \(k_s\) be a separable closure of a field \(k\).
Let \(T\) be an algebraic torus over \(k\). The group
\[
H^1(k,T)
\]
classifies \(T\)-torsors over \(k\).
By a \(k\)-form of a \(k\)-variety
\(X_0\), we mean a \(k\)-variety \(X\) such that
\[
X_{k_s}\simeq (X_0)_{k_s}.
\]
The \(k\)-forms of a $T$-toric variety are in bijection with $H^1(k,T)$.
This correspondence is used, for example, in
Blunk's classification of degree-six del Pezzo surfaces up to isomorphism
preserving the torus action \cite[Theorem~2.4]{BLUNK}; see also
\cite{Duncan2016} for the general framework of twisted forms of toric varieties.

Suppose \(T\) is split by a finite Galois extension \(E/k\).
Following \cite{retforms}, a \emph{(relative) Brauer presentation}
of \(H^1(k,T)\) is an isomorphism
\[
H^1(k,T)
\cong
\ker\!\left(
\bigoplus_i \Br(E/F_i)
\longrightarrow
\bigoplus_j \Br(E/R_j)
\right),
\]
where the map is built from restriction, corestriction, and conjugation
(see \S~\ref{brauerparametrization} {below}).
Here \(F_i\) and \(R_j\) are intermediate fields of \(E/k\), and
\[
\Br(E/F):=\ker\bigl(\Br(F)\to\Br(E)\bigr).
\]
We call the fields \(F_i\) the \emph{source fields}
and the fields \(R_j\) the \emph{target fields}.

This paper studies Brauer presentations for the tori associated with a specific
common lattice. For each even \(n\), let \(V_n\) denote the split generalized
del Pezzo toric variety over \(k\), and let \(L_{n+1}\) denote the split
Losev--Manin space. We use this indexing so that both varieties are
\(n\)-dimensional: \(V_n\) has dimension \(n\), and \(L_{n+1}\) is the toric
variety associated with the Weyl fan of type \(A_n\). In the toric category,
the torus action is part of the structure; for a toric \(T\)-variety over
\(k\), the dense open orbit is a \(T\)-torsor, and it is identified with \(T\)
only when that open orbit has a \(k\)-point.

Let \(G:=S_{n+1}\times S_2\). We consider the \(G\)-lattice

\[
M=\{(a_0,\dots,a_n)\in \mathbb Z^{n+1}\mid a_0+\cdots+a_n=0\},
\]
where
\(S_{n+1}\) acts by permuting the coordinates and the nontrivial element of
\(S_2\) acts by multiplication by \(-1\). In the main results of this paper, we
restrict to \(k\)-tori \(T\) whose character lattice over \(k_s\) is identified
with \(M\), and whose minimal Galois splitting field \(E/k\) satisfies
\(\operatorname{Gal}(E/k)\cong G\), with this isomorphism realizing the
specified \(G\)-action on \(M\).

By Lamarche's description of the split generalized del Pezzo variety \(V_n\),
this lattice is the character lattice of the torus acting on \(V_n\)
\cite[Lemma~2.3.3]{Lamarche}. The same lattice also appears as the character
lattice of the torus acting on the split Losev--Manin space \(L_{n+1}\), whose
fan is the Weyl fan of type \(A_n\) \cite[\S 2]{LosevManin2000}. Thus the Brauer
computations below apply to this common torus lattice in the case where the
Galois action realizes the full group \(G\).

Under the assumption \(\operatorname{Gal}(E/k)\cong G\), the field \(E\)
contains two distinguished subextensions: the quadratic extension
\(K=E^{S_{n+1}\times\{1\}}\) and the degree \(n+1\) extension
\(L=E^{S_n\times S_2}\), where \(S_n\leq S_{n+1}\) is the stabilizer of one
point. We now state the Brauer parametrization in this setting.

\medskip

\begin{theorem}[Brauer parametrization; see Theorem~\ref{thm:gen-dPn-relative-brauer}]
With notation as above,
\[
H^1(k,T)
\cong
\ker\!\left(
\Br(E/K)\oplus \Br(E/L)
\longrightarrow
\Br(E/KL)\oplus \Br(E/k)^{\oplus 2}
\right),
\]
where
\[
(B,Q)\longmapsto
\Bigl(
\operatorname{Res}_{KL/K}(B)+\operatorname{Res}_{KL/L}(Q),\,
\operatorname{Cor}_{K/k}(B),\,
\operatorname{Cor}_{L/k}(Q)
\Bigr).
\]
\end{theorem}

The two coordinates in this theorem are Lamarche's Brauer invariants. Lamarche
constructs the injective map
\[
H^1(k,T)\longrightarrow \Br(K)\oplus \Br(L)
\]
\cite[Definition~2.3.1]{Lamarche}, but does not determine its image. The theorem
above gives the missing image description: it characterizes exactly which pairs of
Brauer classes arise from \(T\)-torsors. For \(n=2\), this specializes to the
Brauer-theoretic description in Blunk's classification of degree-six del Pezzo
surfaces \cite{BLUNK}.

\medskip
\begin{definition}[Admissible source fields]\label{def:admissible-source-fields}
Let \(\widetilde L\subseteq E\) be the Galois closure of \(L/k\).

\begin{enumerate}[label=\textup{(\roman*)}]
\item We say that \(F\) is \emph{\(2\)-admissible} if \(F\) lies in \(\widetilde L\),
there exist two distinct \(k\)-conjugates \(L'\) and \(L''\) of \(L\), and there
exists a maximal odd-degree subextension \(F^\sharp\) of
\(\widetilde L/L'L''\) containing \(FL'L''\), such that the smallest field \(R\)
with \(F\subseteq R\subseteq F^\sharp\) and \(F^\sharp/R\) Galois satisfies
\(L'\subseteq R\).

\item For an odd prime \(p\mid(n+1)\), we say that \(F\) is \emph{\(p\)-admissible}
if \(K\subseteq F\) and \(p\nmid [F:K]\).
\end{enumerate}
\end{definition}

\begin{theorem}[Field-theoretic degree-zero criterion; see Corollary~\ref{cor:field-theoretic-degree-zero}]
There exists a relative Brauer presentation
of \(H^1(k,T)\) with source fields
$F_1,\dots,F_r$
if and only if:
\begin{enumerate}[label=\textup{(\roman*)}]
\item at least one \(F_i\) is \(2\)-admissible;

\item for every odd prime \(p\mid(n+1)\), at least one \(F_i\) is
\(p\)-admissible.
\end{enumerate}
\end{theorem}

This theorem determines exactly which intermediate fields of \(E/k\) can appear
on the source side of a relative Brauer presentation of \(H^1(k,T)\). The fields
on the target side, and the restriction and corestriction map between the two
sides, depend on the particular presentation chosen.

The paper is organized as follows. We first recall the toric description of the
generalized del Pezzo variety \(V_n\), the Losev-Manin space, and the common
character lattice \(M\). We then analyze the relevant permutation lattices
locally at each prime. Next, we compute a completed minimal projective
presentation of the cohomological Mackey functor \(H^1(-,M).\) Using the framework of \cite{retforms}, we translate this presentation into the Brauer parametrization of \(H^1(k,T)\). Finally, we determine which degree-zero permutation terms can occur and translate that criterion into the field-theoretic source-field criterion stated above.

\emph{Acknowledgements.}
I am grateful to my advisor, Alexander Duncan, for many helpful discussions,
suggestions, and encouragement throughout this project. I also thank Asher Auel
and Brendan Hassett for raising questions that motivated this work.

\section{Generalized del Pezzo varieties and the Losev--Manin space}
\label{sec:delpezzo-losev-manin}

Throughout this section \(n\) is even. Set \(\Omega:=\{0,\dots,n\}\) and
\(G:=S_{n+1}\times S_2\). The group \(S_{n+1}\) acts on \(\Omega\), and we
identify \(S_n\le S_{n+1}\) with the stabilizer of the point \(n\).

This section recalls the two split toric varieties that motivate the paper: the
generalized del Pezzo variety \(V_n\) and the Losev--Manin space \(L_{n+1}\).
They are different toric compactifications, but the tori acting on them have the
same \(G\)-lattice as character lattice. 

\subsection{Generalized del Pezzo varieties}

Centrally symmetric smooth toric Fano varieties were classified by
Voskresenskii and Klyachko \cite{Voskresenskii_Klyachko_1985}. In particular,
such varieties are products of projective lines and generalized del Pezzo
varieties \(V_n\) of even dimension \(n\).

We use the following description of \(V_n\), following
Ballard--Duncan--McFaddin \cite{BallardDuncanMcFaddin2022} and Lamarche
\cite[Definition~2.2.1]{Lamarche}. Let \(N\simeq \mathbb Z^n\) have standard
basis \(e_1,\dots,e_n\), and set \(e_0:=-e_1-\cdots-e_n\). The split
generalized del Pezzo variety \(V_n\) is the toric variety whose rays are
generated by \(e_0,e_1,\dots,e_n\) and their opposites
\(\overline e_i:=-e_i\) for \(0\le i\le n\). Its maximal cones are generated by
sets \(\{e_i\}_{i\in A}\cup\{\overline e_i\}_{i\in B}\), where
\(A,B\subset\Omega\) are disjoint and \(|A|=|B|=n/2\).

The fan of \(V_n\) has a natural \(G\)-action: \(S_{n+1}\) permutes the indices
in \(\Omega=\{0,\dots,n\}\), and the nontrivial element of \(S_2\) exchanges
\(e_i\) and \(\overline e_i\). Lamarche shows that the automorphism group fits
into a split exact sequence
\[
1\longrightarrow T\longrightarrow \operatorname{Aut}(V_n)
\longrightarrow S_{n+1}\times S_2\longrightarrow 1,
\]
where \(T\) is the dense torus of \(V_n\)
\cite[Proposition~2.3.1]{Lamarche}. Thus, for \(V_n\), the toric automorphism
group is the full automorphism group.

\subsection{The Losev-Manin space}

The Losev-Manin space gives another toric compactification with the same dense
torus. Recall that
\[
M=
\left\{
(a_0,\dots,a_n)\in \mathbb Z^{n+1}
\;\middle|\;
a_0+\cdots+a_n=0
\right\},
\]
and put \(N:=\operatorname{Hom}(M,\mathbb Z)\). For every nonempty proper subset
\(I\subsetneq \Omega\), define \(v_I\in N\) by
\[
v_I(a_0,\dots,a_n)=\sum_{i\in I}a_i.
\]
Let \(\Sigma_n\) be the fan in \(N_\mathbb R\) whose cones are generated by
chains of such subsets:
\[
\sigma_{I_\bullet}
=
\mathbb R_{\ge 0}v_{I_1}
+\cdots+
\mathbb R_{\ge 0}v_{I_r},
\qquad
\emptyset\subsetneq I_1\subsetneq\cdots\subsetneq I_r\subsetneq \Omega.
\]
Equivalently, the maximal cones are indexed by permutations
\(\pi\in S_{n+1}\) and are generated by
\[
v_{\{\pi(0)\}},\;
v_{\{\pi(0),\pi(1)\}},\;
\dots,\;
v_{\{\pi(0),\dots,\pi(n-1)\}}.
\]
This is the Weyl, or permutohedral, fan of type \(A_n\).\\
The associated toric variety \(X(\Sigma_n)\) is the split Losev--Manin space \(L_{n+1}\). We use the standard toric
realization of the Losev-Manin moduli space introduced by Losev and Manin
\cite[\S 2]{LosevManin2000}; see also Batyrev-Blume for the realization via
Weyl fans of root systems \cite{Batyrev_Blume_2011}. The same space appears in
Hassett's theory of weighted pointed stable curves as the moduli space with two
heavy marked points and \(n+1\) light marked points
\cite[\S 6.4]{HassettWeighted}.

The dense torus of \(L_{n+1}\) has character lattice \(M\). The fan
\(\Sigma_n\) is preserved by the \(G=S_{n+1}\times S_2\)-action already fixed on
\(M\): the group \(S_{n+1}\) sends \(v_I\) to \(v_{gI}\), and the nontrivial
element of \(S_2\) sends \(v_I\) to \(v_{\Omega\setminus I}=-v_I\). For
\(n+1\ge 3\), this gives \(\operatorname{Aut}(\Sigma_n)\cong S_{n+1}\times S_2.\)
Thus \(V_n\) and \(L_{n+1}\) are different toric compactifications whose dense
tori have the same \(G\)-lattice as character lattice.

\subsection{Toric forms and torsors}

Let \(X(\Sigma)\) be a split toric variety with split torus \(T_0\). A toric
form has two pieces of descent data: a finite Galois action on the fan
\(\Sigma\), and a torsor under the corresponding twisted torus.

Equivalently, toric forms are governed by the non-abelian cohomology set
\(H^1(k,T_0\rtimes \operatorname{Aut}(\Sigma))\). Projection to
\(\operatorname{Aut}(\Sigma)\) records the fan action. If
\([\rho]\in H^1(k,\operatorname{Aut}(\Sigma))\) is fixed, and \(T_\rho\) is the
corresponding twisted torus, then the fibre over \(\rho\) is
\[
H^1(k,T_\rho)
\big/
H^0\!\bigl(k,{}^\rho\operatorname{Aut}(\Sigma)\bigr).
\]
Thus \(H^1(k,T_\rho)\) is the torsor part of the classification, while
\(H^0(k,{}^\rho\operatorname{Aut}(\Sigma))\) accounts for residual automorphisms
of the twisted fan.

If the torus action is part of the structure, then \(H^1(k,T_\rho)\) itself is
the relevant torsor group. This is the sense in which \(H^1(k,T)\) classifies
forms as \(T\)-toric varieties. This viewpoint is used in Blunk's classification
of degree-six del Pezzo surfaces up to isomorphism preserving the torus action
\cite[Theorem~2.4]{BLUNK}; see also Duncan's framework for twisted forms of
toric varieties \cite{Duncan2016}.

\begin{remark}[Abstract forms]
For the two compactifications considered here, the full automorphism group is
the toric automorphism group. For generalized del Pezzo varieties, Lamarche
proves that there is a split exact sequence
\[
1\longrightarrow T\longrightarrow \operatorname{Aut}(V_n)
\longrightarrow S_{n+1}\times S_2\longrightarrow 1
\]
\cite[Proposition~2.3.1]{Lamarche}. For the Losev--Manin space
\(L_{n+1}=X(\Sigma_n)\), one likewise has
\[
\operatorname{Aut}(L_{n+1})
\cong
T\rtimes \operatorname{Aut}(\Sigma_n)
\cong
T\rtimes(S_{n+1}\times S_2)
\]
for \(n+1\ge 3\). Hence, for both \(V_n\) and \(L_{n+1}\), the classification of
abstract forms agrees with the classification of toric forms. The difference
between the two varieties is the choice of toric compactification, not the
underlying torus or its finite automorphism group.
\end{remark}

\subsection{Standing notation}

Let \(\sigma\) denote the nontrivial element of \(S_2\). We write
\[
\mathbb Z[S_{n+1}/S_n]=\bigoplus_{i=0}^n \mathbb Z u_i,
\]
where \(u_i\) corresponds to the point \(i\in\Omega\). This is regarded as a
\(G\)-lattice through the projection \(G\to S_{n+1}\), so the \(S_2\)-factor
acts trivially. Similarly, \(\mathbb Z[S_2]=\mathbb Z\{1,\sigma\}\) is regarded
as a \(G\)-lattice through the projection \(G\to S_2\), so the
\(S_{n+1}\)-factor acts trivially. With \(S_n\) embedded in \(G\) as
\(S_n\times 1\), we have a natural \(G\)-lattice isomorphism
\[
\mathbb Z[G/(S_n\times 1)]
\cong
\mathbb Z[S_{n+1}/S_n]\otimes_{\mathbb Z}\mathbb Z[S_2].
\]

Let \(\mathbb Z^{-}\) denote the rank-one \(G\)-lattice on which \(S_{n+1}\)
acts trivially and \(\sigma\) acts by \(-1\). 

For a finite \(G\)-set \(X\), let \(\operatorname{aug}_X:\mathbb Z[X]\to
\mathbb Z\) denote the augmentation map. In particular,
\(\operatorname{aug}_{n+1}:\mathbb Z[S_{n+1}/S_n]\to\mathbb Z\) is given by
\[
\sum_{i=0}^n a_i u_i\longmapsto \sum_{i=0}^n a_i,
\]
and \(\operatorname{aug}_2:\mathbb Z[S_2]\to\mathbb Z\) is given by
\[
a\cdot 1+b\cdot \sigma\longmapsto a+b.
\]
Thus
\[
\ker(\operatorname{aug}_2)\cong \mathbb Z^{-},
\qquad
1-\sigma \longleftrightarrow 1.
\]
For a commutative ring \(R\), set \(R^{-}:=\mathbb Z^{-}\otimes_{\mathbb Z}R\).

\subsection{The integral permutation sequence}

Define
\[
d_1:
\mathbb Z[S_{n+1}/S_n]\otimes_{\mathbb Z}\mathbb Z[S_2]
\longrightarrow
\mathbb Z[S_{n+1}/S_n]\oplus \mathbb Z[S_2]
\]
by
\[
d_1=
(\operatorname{id}\otimes \operatorname{aug}_2,\,
\operatorname{aug}_{n+1}\otimes \operatorname{id}),
\]
and define \(d_2:\mathbb Z[S_{n+1}/S_n]\oplus\mathbb Z[S_2]\to\mathbb Z\) by
\(d_2=\operatorname{aug}_{n+1}-\operatorname{aug}_2\).

The following sequence is the split case of Lamarche's permutation-lattice
presentation of the character lattice of the dense torus of \(V_n\)
\cite[Lemma~2.3.3]{Lamarche}.

\begin{prop}[Lamarche's integral permutation sequence]
\label{prop:M-description}
There is an exact sequence of \(G\)-lattices
\begin{equation}\label{eq:Lamarche}
0 \longrightarrow M \longrightarrow
\mathbb Z[S_{n+1}/S_n]\otimes_{\mathbb Z}\mathbb Z[S_2]
\xrightarrow{d_1}
\mathbb Z[S_{n+1}/S_n]\oplus \mathbb Z[S_2]
\xrightarrow{d_2}
\mathbb Z
\longrightarrow 0.
\end{equation}
\end{prop}

\begin{proof}
Let \(A:=\mathbb Z[S_{n+1}/S_n]\), \(B:=\mathbb Z[S_2]\),
\(A^0:=\ker(\operatorname{aug}_{n+1})\), and
\(B^0:=\ker(\operatorname{aug}_2)\). Thus \(M=A^0\otimes_{\mathbb Z}B^0\).

We first identify the kernel of \(d_1\). Choose the basis elements \(u_n\in A\) and \(1\in B\). Since
\(\operatorname{aug}_{n+1}(u_n)=1\) and \(\operatorname{aug}_2(1)=1\), the
augmentation sequences split as abelian groups:
\[
A=A^0\oplus \mathbb Z u_n,
\qquad
B=B^0\oplus \mathbb Z\cdot 1.
\] Hence
\[
A\otimes B
=
(A^0\otimes B^0)
\oplus
(A^0\otimes \mathbb Z\cdot 1)
\oplus
(\mathbb Z u_n\otimes B^0)
\oplus
(\mathbb Z u_n\otimes \mathbb Z\cdot 1).
\]
If \(x\in A\otimes B\), write
\(x=x_{00}+x_{10}\otimes 1+u_n\otimes x_{01}+c\,u_n\otimes 1\), with
\(x_{00}\in A^0\otimes B^0\), \(x_{10}\in A^0\), \(x_{01}\in B^0\), and
\(c\in\mathbb Z\). Then \(d_1(x)=(x_{10}+c u_n,\;x_{01}+c\cdot 1)\).
Thus \(d_1(x)=0\) implies \(c=0\), \(x_{10}=0\), and \(x_{01}=0\). Therefore
\(\ker(d_1)=A^0\otimes B^0=M\).

Next, \(d_2\circ d_1=0\) by definition. Conversely, suppose \((a,b)\in A\oplus
B\) satisfies \(d_2(a,b)=0\). Then
\(\operatorname{aug}_{n+1}(a)=\operatorname{aug}_2(b)\). Let this common integer
be \(m\). Then \(a-mu_n\in A^0\) and \(b-m\cdot 1\in B^0\). The element
\((a-mu_n)\otimes 1+u_n\otimes b\in A\otimes B\) maps under \(d_1\) to
\((a,b)\). Hence \(\ker(d_2)=\operatorname{im}(d_1)\). Finally, \(d_2\) is
surjective, since \(d_2(u_n,0)=1\). This proves exactness.
\end{proof}

By Lamarche's description of the torus acting on \(V_n\), the lattice \(M\) is
the character lattice \(X^*(T)\) of the dense torus
\cite[Lemma~2.3.3]{Lamarche}. \\

\section{Local structure of the permutation lattices}\label{gendell}
If \(R\) is a commutative ring, we write
\(M_R:=M\otimes_{\mathbb Z}R\) and for a prime \(p\), we write
\(M_p:=M\otimes_{\mathbb Z}\mathbb Z_p\).
set
\[
U_p:=\mathbb Z_p[S_{n+1}/S_n],
\qquad
A_p:=\ker\!\left(
\operatorname{aug}_{n+1}:U_p\to \mathbb Z_p
\right).
\]

\begin{lemma}[Farahat]
\label{lem:Farahat-natural}
Let \(p\) be a prime.
Then \(A_p\) is indecomposable and
\[
\operatorname{End}_{\mathbb Z_pS_{n+1}}(A_p)\cong \mathbb Z_p.
\]
Moreover, the augmentation sequence
\[
0\longrightarrow A_p
\longrightarrow U_p
\xrightarrow{\operatorname{aug}_{n+1}}
\mathbb Z_p
\longrightarrow 0
\]
splits if and only if \(p\nmid(n+1)\). If \(p\mid(n+1)\), then \(U_p\) is
indecomposable.
\end{lemma}

\begin{proof}
This is Farahat's result on the natural representation of the symmetric group,
applied to the natural permutation lattice
\[
\mathbb Z_p[S_{n+1}/S_n].
\]
See \cite[Theorem~2.1, Lemma~4.7, Corollary~4.8, and
Theorem~4.12]{Farahat_1962}.
\end{proof}
\begin{lemma}\label{lem:odd-Up-and-S2}
Let \(p\) be an odd prime dividing \(n+1\).
Then \(U_p\) is indecomposable as a \(\mathbb Z_pS_{n+1}\)-lattice. Moreover,
\[
\mathbb Z_p[S_2]\cong \mathbb Z_p\oplus \mathbb Z_p^{-}.
\]
\end{lemma}

\begin{proof}
The first assertion follows from Lemma~\ref{lem:Farahat-natural}. Since \(p\mid(n+1)\), the natural permutation lattice
\[
U_p=\mathbb Z_p[S_{n+1}/S_n]
\]
is indecomposable.

For the second assertion, since \(p\) is odd, \(2\in\mathbb Z_p^\times\).
Thus the idempotents
\[
\frac{1+\sigma}{2},
\qquad
\frac{1-\sigma}{2}
\]
split \(\mathbb Z_p[S_2]\) into its trivial and sign summands:
\[
\mathbb Z_p[S_2]\cong \mathbb Z_p\oplus \mathbb Z_p^{-}.
\]
\end{proof}
\begin{prop}[Odd-primary summand]\label{prop:odd-short-exact-sequence}
Let \(p\) be an odd prime dividing \(n+1\).
Then the \(p\)-localization of \eqref{eq:Lamarche} contains, as a direct
summand, the short exact sequence
\[
0\longrightarrow M_p
\longrightarrow U_p\otimes_{\mathbb Z_p}\mathbb Z_p^-
\xrightarrow{\operatorname{aug}_{n+1}\otimes 1}
\mathbb Z_p^-
\longrightarrow 0.
\]
\end{prop}

\begin{proof}
After tensoring \eqref{eq:Lamarche} with \(\mathbb Z_p\), we use the splitting
\[
\mathbb Z_p[S_2]\cong \mathbb Z_p\oplus \mathbb Z_p^-
\]
from Lemma~\ref{lem:odd-Up-and-S2}. Taking the sign summand gives the map
\[
U_p\otimes_{\mathbb Z_p}\mathbb Z_p^-
\xrightarrow{\operatorname{aug}_{n+1}\otimes 1}
\mathbb Z_p^-.
\]
Its kernel is
\[
\ker(\operatorname{aug}_{n+1})
\otimes_{\mathbb Z_p}
\mathbb Z_p^-.
\]
By Proposition~\ref{prop:M-description}, this kernel is precisely \(M_p\).
Hence the displayed short exact sequence is a direct summand of the
\(p\)-localization of \eqref{eq:Lamarche}.
\end{proof}
Recall from the above 
\[
A_2:=\ker\!\left(
\operatorname{aug}_{n+1}:
\mathbb Z_2[S_{n+1}/S_n]\longrightarrow \mathbb Z_2
\right).
\]
\begin{lemma}[The prime \(2\)]\label{lem:two-local-lattices}

Then \(A_2\) is indecomposable of rank \(n\), and
\[
\mathbb Z_2[S_{n+1}/S_n]
\cong
\mathbb Z_2\oplus A_2.
\]
Moreover, \(\mathbb Z_2[S_2]\) is indecomposable, and
\[
A_2\otimes_{\mathbb Z_2}\mathbb Z_2[S_2]
\]
is indecomposable of rank \(2n\). Finally,
\[
\mathbb Z_2\bigl[(S_{n+1}\times S_2)/(S_n\times 1)\bigr]
\cong
\mathbb Z_2[S_2]\oplus
\bigl(A_2\otimes_{\mathbb Z_2}\mathbb Z_2[S_2]\bigr).
\]
\end{lemma}

\begin{proof}
Since \(n+1\) is odd, the integer \(n+1\) is a unit in \(\mathbb Z_2\).
Hence the augmentation sequence
\[
0\longrightarrow A_2\longrightarrow
\mathbb Z_2[S_{n+1}/S_n]
\xrightarrow{\operatorname{aug}_{n+1}}
\mathbb Z_2
\longrightarrow 0
\]
splits. By Lemma~\ref{lem:Farahat-natural}, \(A_2\) is indecomposable and
\[
\operatorname{End}_{\mathbb Z_2S_{n+1}}(A_2)\cong \mathbb Z_2.
\]
Thus
\[
\mathbb Z_2[S_{n+1}/S_n]
\cong
\mathbb Z_2\cdot(u_0+\cdots+u_n)\oplus A_2 \cong \mathbb Z_2 \oplus A_2.
\]

The lattice \(\mathbb Z_2[S_2]\) is indecomposable because
\(\mathbb Z_2[S_2]\) is the group algebra of a \(2\)-group over the local ring
\(\mathbb Z_2\), hence is local.

Next,
\[
\operatorname{End}_{\mathbb Z_2G}
\left(A_2\otimes_{\mathbb Z_2}\mathbb Z_2[S_2]\right)
\cong
\operatorname{End}_{\mathbb Z_2S_{n+1}}(A_2)
\otimes_{\mathbb Z_2}
\operatorname{End}_{\mathbb Z_2S_2}(\mathbb Z_2[S_2]).
\]
Using
\[
\operatorname{End}_{\mathbb Z_2S_{n+1}}(A_2)\cong \mathbb Z_2
\]
and
\[
\operatorname{End}_{\mathbb Z_2S_2}(\mathbb Z_2[S_2])
\cong \mathbb Z_2[S_2],
\]
we get
\[
\operatorname{End}_{\mathbb Z_2G}
\left(A_2\otimes_{\mathbb Z_2}\mathbb Z_2[S_2]\right)
\cong
\mathbb Z_2[S_2].
\]
This ring is local, so
\[
A_2\otimes_{\mathbb Z_2}\mathbb Z_2[S_2]
\]
is indecomposable.

Finally,
\[
\mathbb Z_2\bigl[(S_{n+1}\times S_2)/(S_n\times 1)\bigr]
\cong
\mathbb Z_2[S_{n+1}/S_n]\otimes_{\mathbb Z_2}\mathbb Z_2[S_2].
\]
Using the decomposition of \(\mathbb Z_2[S_{n+1}/S_n]\), this becomes
\[
\mathbb Z_2[S_2]\oplus
\bigl(A_2\otimes_{\mathbb Z_2}\mathbb Z_2[S_2]\bigr).
\]
\end{proof}

\begin{prop}[The \(2\)-primary summand]\label{prop:two-short-exact-sequence}
The \(2\)-localization of \eqref{eq:Lamarche} contains, as a direct summand,
the short exact sequence
\[
0\longrightarrow M_2
\longrightarrow A_2\otimes_{\mathbb Z_2}\mathbb Z_2[S_2]
\xrightarrow{1_{A_2}\otimes \operatorname{aug}_2}
A_2
\longrightarrow 0.
\]
\end{prop}

\begin{proof}
Tensor \eqref{eq:Lamarche} with \(\mathbb Z_2\). By
Lemma~\ref{lem:two-local-lattices},
\[
\mathbb Z_2[S_{n+1}/S_n]
\cong
\mathbb Z_2\oplus A_2.
\]
Taking the \(A_2\)-summand of the localized sequence gives the map
\[
A_2\otimes_{\mathbb Z_2}\mathbb Z_2[S_2]
\xrightarrow{1_A\otimes \operatorname{aug}_2}
A_2.
\]
Its kernel is
\[
A_2\otimes_{\mathbb Z_2}\ker(\operatorname{aug}_2).
\]
Since
\[
A_2=\ker(\operatorname{aug}_{n+1})
\]
over \(\mathbb Z_2\), Proposition~\ref{prop:M-description} gives
\[
M_2\cong
A_2\otimes_{\mathbb Z_2}\ker(\operatorname{aug}_2).
\]
Therefore the \(A_2\)-summand of the localized sequence is precisely
\[
0\longrightarrow M_2
\longrightarrow A_2\otimes_{\mathbb Z_2}\mathbb Z_2[S_2]
\xrightarrow{1_{A_2}\otimes \operatorname{aug}_2}
A_2
\longrightarrow 0.
\]
\end{proof}

\section{Minimal projective presentations of the Mackey functor
\texorpdfstring{$H^1(-, M)$}{}}

We now pass from the \(G\)-lattice \(M\) to the cohomological Mackey functor
\[
F:=H^1(-,M).
\]
Thus, for every subgroup \(H\le G\),
\[
F(H)=H^1(H,M),
\]
with the usual restriction, corestriction, and conjugation maps. More generally,
for a coefficient ring \(R\), we write
\[
F_R:=H^1(-,M_R),
\qquad
F_R(H)=H^1(H,M_R).
\]

For an \(RG\)-lattice \(N\), we write
\[
FP_N:=H^0(-,N)
\]
for the associated fixed-point cohomological Mackey functor. In this section we
recall the necessary terminology on cohomological Mackey functors and compute
the beginning of the minimal projective presentation of \(F_R\), first
prime-locally and then over \(\widehat{\mathbb Z}\).

For background on Mackey functors, projective presentations in the category of
cohomological Mackey functors, and their relation to \(G\)-lattices, see
\cite{T-W,Thevenaz_Webb_1990,Webb_2000}.\\
By the Brauer-realization theorem of Duncan--Singh \cite{retforms},
permutation-projective presentations of \(F\) give relative Brauer presentations
of \(H^1(k,T)\). In this dictionary, degree-zero permutation
summands give source fields, while the first differential gives the
restriction, corestriction, and conjugation relations among the corresponding
Brauer classes. This is why the Mackey-functor presentation computed below is
the input for the Brauer parametrization in Section~\ref{brauerparametrization}.

\begin{lemma}\label{lem:H1-vanishing-permutation-summands}
Let \(p\) be a prime, and let \(N\) be a direct summand of a permutation
\(\mathbb Z_pG\)-lattice. Then
\[
H^1(\Gamma,N)=0
\]
for every subgroup \(\Gamma\le G\).
\end{lemma}

\begin{proof}
It suffices to prove the claim for a transitive permutation lattice. After
restriction to \(\Gamma\), such a lattice decomposes as a direct sum of
lattices of the form
\[
\mathbb Z_p[\Gamma/\Delta],
\qquad
\Delta\le \Gamma.
\]
By Shapiro's lemma,
\[
H^1(\Gamma,\mathbb Z_p[\Gamma/\Delta])
\cong
H^1(\Delta,\mathbb Z_p).
\]
Since \(\Delta\) is finite and \(\mathbb Z_p\) is torsion-free,
\[
H^1(\Delta,\mathbb Z_p)=\operatorname{Hom}(\Delta,\mathbb Z_p)=0.
\]
The result follows for direct sums and hence for direct summands.
\end{proof}

Let
\[
\widehat{\mathbb Z}:=\prod_p\mathbb Z_p,
\qquad
M_{\widehat{\mathbb Z}}:=M\otimes_{\mathbb Z}\widehat{\mathbb Z},
\]
and set
\[
F_{\widehat{\mathbb Z}}:=H^1(-,M_{\widehat{\mathbb Z}}).
\]

\begin{theorem}[Completed minimal projective presentation]
\label{thm:completed-minimal-projective-presentation}
The beginning of the minimal projective presentation of
\(F_{\widehat{\mathbb Z}}\) in
\(\operatorname{CoMack}_{\widehat{\mathbb Z}}(G)\) is
\[
FP_{\mathcal P_1}
\xrightarrow{\;FP_\delta\;}
FP_{\mathcal P_0}
\longrightarrow
F_{\widehat{\mathbb Z}}
\longrightarrow 0,
\]
where
\[
\mathcal P_0
=
A_2
\oplus
\bigoplus_{\substack{p\mid(n+1)\\ p\ \mathrm{odd}}}
\mathbb Z_p^-,
\]
and
\[
\mathcal P_1
=
\left(A_2\otimes_{\mathbb Z_2}\mathbb Z_2[S_2]\right)
\oplus
\bigoplus_{\substack{p\mid(n+1)\\ p\ \mathrm{odd}}}
\left(U_p\otimes_{\mathbb Z_p}\mathbb Z_p^-\right).
\]
Each local summand is regarded as a \(\widehat{\mathbb Z}G\)-module via the
projection \(\widehat{\mathbb Z}\to\mathbb Z_p\). The map
\[
\delta:\mathcal P_1\longrightarrow \mathcal P_0
\]
is the direct sum of
\[
1_{A_2}\otimes \operatorname{aug}_2:
A_2\otimes_{\mathbb Z_2}\mathbb Z_2[S_2]\longrightarrow A_2
\]
at \(p=2\), and
\[
\operatorname{aug}_{n+1}\otimes 1:
U_p\otimes_{\mathbb Z_p}\mathbb Z_p^-
\longrightarrow
\mathbb Z_p^-
\]
for each odd prime \(p\mid(n+1)\).
\end{theorem}

\begin{proof}
At \(p=2\), Proposition~\ref{prop:two-short-exact-sequence} gives
\[
0\longrightarrow M_2
\longrightarrow
A_2\otimes_{\mathbb Z_2}\mathbb Z_2[S_2]
\xrightarrow{1_{A_2}\otimes \operatorname{aug}_2}
A_2
\longrightarrow 0.
\]
The middle term is a direct summand of a permutation lattice, so by
Lemma~\ref{lem:H1-vanishing-permutation-summands}, applying fixed points gives
an exact sequence of cohomological Mackey functors
\[
FP_{A_2\otimes_{\mathbb Z_2}\mathbb Z_2[S_2]}
\longrightarrow
FP_{A_2}
\longrightarrow
H^1(-,M_2)
\longrightarrow 0.
\]

For an odd prime \(p\mid(n+1)\), Proposition~\ref{prop:odd-short-exact-sequence}
gives
\[
0\longrightarrow M_p
\longrightarrow
U_p\otimes_{\mathbb Z_p}\mathbb Z_p^-
\xrightarrow{\operatorname{aug}_{n+1}\otimes 1}
\mathbb Z_p^-
\longrightarrow 0.
\]
Again the middle term is a direct summand of a permutation lattice, and hence
\[
FP_{U_p\otimes_{\mathbb Z_p}\mathbb Z_p^-}
\longrightarrow
FP_{\mathbb Z_p^-}
\longrightarrow
H^1(-,M_p)
\longrightarrow 0
\]
is exact.

If \(p\) is odd and \(p\nmid(n+1)\), then \(M_p\) is a direct summand of a
permutation \(\mathbb Z_pG\)-lattice, so
\[
H^1(-,M_p)=0
\]
by Lemma~\ref{lem:H1-vanishing-permutation-summands}. Thus the only nonzero
local contributions occur at \(p=2\) and at the odd primes dividing \(n+1\).

Assembling the nonzero local presentations over
\[
\widehat{\mathbb Z}=\prod_p\mathbb Z_p
\]
gives the displayed \(\widehat{\mathbb Z}\)-projective presentation. Indeed, by
\cite[Theorem~16.5]{T-W}, fixed-point functors attached to direct summands of
permutation lattices are projective cohomological Mackey functors.

It remains to check minimality, which may be done componentwise after applying
the projections
\[
\widehat{\mathbb Z}\longrightarrow \mathbb Z_p.
\]
At \(p=2\), Lemma~\ref{lem:two-local-lattices} shows that
\[
A_2
\quad\text{and}\quad
A_2\otimes_{\mathbb Z_2}\mathbb Z_2[S_2]
\]
are indecomposable direct summands of permutation lattices. For odd
\(p\mid(n+1)\), Lemma~\ref{lem:Farahat-natural} and the splitting
\[
\mathbb Z_p[S_2]\cong \mathbb Z_p\oplus\mathbb Z_p^-
\]
show that
\[
\mathbb Z_p^-
\quad\text{and}\quad
U_p\otimes_{\mathbb Z_p}\mathbb Z_p^-
\]
are indecomposable direct summands of permutation lattices.

Thus the fixed-point functors appearing in the local presentations above are
indecomposable projectives. Hence the local surjections
\[
FP_{A_2}\longrightarrow H^1(-,M_2)
\qquad\text{and}\qquad
FP_{\mathbb Z_p^-}\longrightarrow H^1(-,M_p)
\]
are projective covers, by the standard right-minimality criterion for epimorphisms
from indecomposable projectives with local endomorphism rings. The same argument
applies to the surjections from the degree-one terms onto the corresponding
kernels. Therefore each nonzero local presentation is minimal. All other local
components vanish, so the completed presentation is minimal.
\end{proof}
\section{The integral permutation presentation}

We now record an integral degree-zero permutation term which contains all the
local minimal degree-zero terms described above. In particular,  if \(Q\) is a
direct summand of a permutation \(RG\)-lattice, then \(FP_Q\) is projective in
the category of cohomological Mackey functors. Conversely, the projective
objects in this category are of this form; see \cite[Chapter~16]{T-W}.

A \emph{projective presentation} of a cohomological Mackey functor \(\mathcal F\)
is an exact sequence
\[
FP_{Q_1}
\longrightarrow
FP_{Q_0}
\longrightarrow
\mathcal F
\longrightarrow 0
\]
where \(Q_0\) and \(Q_1\) are direct summands of permutation \(RG\)-lattices. If
the lattices \(Q_0\) and \(Q_1\) themselves are permutation lattices, we call this
a \emph{permutation presentation}. Similarly, a \emph{permutation resolution} is
a projective resolution whose terms are fixed-point functors attached to
permutation lattices.

The local minimal projective presentations computed above have degree-zero
terms given by direct summands of permutation lattices. The goal of this section
is to choose a single integral permutation degree-zero and  degree-one terms whose localizations
contain all of these local minimal summands.

\begin{prop}\label{prop:local-degree-zero-cover}
For a prime \(p\), set
\[
F_p:=H^1(-,M_p).
\]
Let \((\mathcal P_0)_p\) denote the \(p\)-primary degree-zero summand from
Theorem~\ref{thm:completed-minimal-projective-presentation}. Then the following
hold.

\begin{enumerate}[label=\textup{(\roman*)}]
\item The \(2\)-primary degree-zero summand
\[
FP_{(\mathcal P_0)_2}=FP_{A_2}
\]
is a direct summand of
\[
FP_{\mathbb Z_2[S_{n+1}/S_n]}.
\]

\item For every odd prime \(p\mid(n+1)\), the \(p\)-primary degree-zero summand
\[
FP_{(\mathcal P_0)_p}=FP_{\mathbb Z_p^-}
\]
is a direct summand of
\[
FP_{\mathbb Z_p[S_2]}.
\]

\item If \(p\) is odd and \(p\nmid(n+1)\), then
\[
F_p=0.
\]
\end{enumerate}
\end{prop}

\begin{proof}
By Theorem~\ref{thm:completed-minimal-projective-presentation}, the nonzero
local degree-zero summands are \(A_2\) at \(p=2\), and \(\mathbb Z_p^-\) at the
odd primes \(p\mid(n+1)\). For odd primes \(p\nmid(n+1)\), the same theorem gives
\[
F_p=0.
\]

At \(p=2\), since \(n+1\) is odd, the augmentation sequence
\[
0\longrightarrow A_2
\longrightarrow
\mathbb Z_2[S_{n+1}/S_n]
\longrightarrow
\mathbb Z_2
\longrightarrow 0
\]
splits. Hence
\[
A_2\mid \mathbb Z_2[S_{n+1}/S_n],
\]
and therefore
\[
FP_{A_2}\mid FP_{\mathbb Z_2[S_{n+1}/S_n]}.
\]

Now let \(p\mid(n+1)\) be odd. Since
\[
\mathbb Z_p[S_2]\cong \mathbb Z_p\oplus \mathbb Z_p^-,
\]
we have
\[
\mathbb Z_p^-\mid \mathbb Z_p[S_2].
\]
Therefore
\[
FP_{\mathbb Z_p^-}\mid FP_{\mathbb Z_p[S_2]}.
\]
This proves the proposition.
\end{proof}

\begin{remark}\label{rem:global-degree-zero-term}
Proposition~\ref{prop:local-degree-zero-cover} shows that
\[
FP_{\mathbb Z[S_2]}
\oplus
FP_{\mathbb Z[S_{n+1}/S_n]}
\]
is a natural integral degree-zero permutation term for
\[
F=H^1(-,M).
\]
It is not locally minimal; rather, after \(p\)-adic completion it contains the
local minimal degree-zero summands from
Theorem~\ref{thm:completed-minimal-projective-presentation}, and it is chosen so
that the first differential can be written integrally and uniformly.
\end{remark}
\subsection{The integral first differential}

Set
\[
Q_1:=
\mathbb Z[S_{n+1}/S_n]\otimes_{\mathbb Z}\mathbb Z[S_2]
\oplus \mathbb Z^{\oplus 2},
\]
and
\[
Q_0:=
\mathbb Z[S_2]\oplus \mathbb Z[S_{n+1}/S_n].
\]
Define a \(G\)-lattice homomorphism
\[
\Psi:Q_1\longrightarrow Q_0
\]
by
\[
\Psi(m,s_1,s_2)
=
\left(
(\operatorname{aug}_{n+1}\otimes 1)(m)+s_1(1+\sigma),\,
(1\otimes \operatorname{aug}_2)(m)+s_2\sum_{i=0}^n u_i
\right).
\]
For a commutative ring \(R\), write
\[
Q_{i,R}:=Q_i\otimes_{\mathbb Z}R,
\qquad
\Psi_R:=\Psi\otimes_{\mathbb Z}R.
\]
For a prime \(p\), we write simply
\[
Q_{i,p}:=Q_{i,\mathbb Z_p},
\qquad
\Psi_p:=\Psi_{\mathbb Z_p}.
\]

For later reference, with respect to the ordered bases
\[
(u_0\otimes 1,\dots,u_n\otimes 1,
u_0\otimes \sigma,\dots,u_n\otimes \sigma
\mid \mathbf 1^{(1)}\mid \mathbf 1^{(2)})
\]
of \(Q_1\), and
\[
(1,\sigma\mid u_0,\dots,u_n)
\]
of \(Q_0\), the map \(\Psi\) is represented by
\begin{equation}\label{eq:matrix-Psi}
\begin{pmatrix}
\mathbf 1_{1\times(n+1)} & 0 & 1 & 0\\
0 & \mathbf 1_{1\times(n+1)} & 1 & 0\\
I_{n+1} & I_{n+1} & 0 & \mathbf 1_{(n+1)\times 1}
\end{pmatrix}.
\end{equation}

\begin{prop}\label{prop:localized-H0-exactness}
For every prime \(p\), the sequence of cohomological Mackey functors
\[
FP_{Q_{1,p}}
\xrightarrow{\;FP_{\Psi_p}\;}
FP_{Q_{0,p}}
\longrightarrow
F_p
\longrightarrow 0
\]
is exact.
\end{prop}

\begin{proof}
We argue prime by prime.

If \(p=2\), the decomposition
\[
\mathbb Z_2[S_{n+1}/S_n]\cong \mathbb Z_2\oplus A_2
\]
decomposes \(\Psi_2\). The non-split direct summand is
\[
A_2\otimes_{\mathbb Z_2}\mathbb Z_2[S_2]
\xrightarrow{\,1_{A_2}\otimes\operatorname{aug}_2\,}
A_2.
\]
By the \(2\)-local minimal presentation, the cokernel of the induced map on
fixed-point functors is \(F_2\). The complementary block is split-surjective,
so it contributes no cokernel.

If \(p\mid(n+1)\) is odd, the decomposition
\[
\mathbb Z_p[S_2]\cong \mathbb Z_p\oplus \mathbb Z_p^-
\]
decomposes \(\Psi_p\). The non-split block is
\[
U_p\otimes_{\mathbb Z_p}\mathbb Z_p^-
\xrightarrow{\operatorname{aug}_{n+1}\otimes 1}
\mathbb Z_p^-.
\]
By the odd \(p\)-local minimal presentation, the cokernel of the induced map on
fixed-point functors is \(F_p\). Again, the complementary block is
split-surjective.

Finally, if \(p\nmid 2(n+1)\), then both \(2\) and \(n+1\) are units in
\(\mathbb Z_p\). Hence both relevant augmentation sequences split, so
\(\Psi_p\) is split-surjective. In this case \(M_p\) is a direct summand of a
permutation lattice, and therefore
\[
F_p=H^1(-,M_p)=0
\]
by Lemma~\ref{lem:H1-vanishing-permutation-summands}. Thus the sequence is
exact for every prime \(p\).
\end{proof}
\begin{theorem}\label{thm:integral-first-differential-H0}
The sequence
\[
FP_{Q_1}
\xrightarrow{\;FP_{\Psi}\;}
FP_{Q_0}
\longrightarrow
F
\longrightarrow 0
\]
is exact.
\end{theorem}

\begin{proof}
Let
\[
C:=\operatorname{coker}(FP_{\Psi}).
\]
Since \(\mathbb Z_p\) is flat over \(\mathbb Z\), and since fixed points of the
finite-rank free lattices \(Q_0\) and \(Q_1\) commute with flat base change, we have
a natural isomorphism of cohomological Mackey functors
\[
C\otimes_{\mathbb Z}\mathbb Z_p
\cong
\operatorname{coker}(FP_{\Psi_p}).
\]
By Proposition~\ref{prop:localized-H0-exactness}, the latter cokernel is
\(F_p=H^1(-,M_p)\). Also, for every subgroup \(U\leq G\), the natural base-change
map gives
\[
H^1(U,M)\otimes_{\mathbb Z}\mathbb Z_p
\cong
H^1(U,M_p).
\]
Hence we obtain a natural isomorphism of cohomological Mackey functors
\[
C\otimes_{\mathbb Z}\mathbb Z_p
\cong
F\otimes_{\mathbb Z}\mathbb Z_p
\]
for every prime \(p\).

Over \(\mathbb Q\), both \(2\) and \(n+1\) are invertible, so the same splitting
argument shows that
\[
C\otimes_{\mathbb Z}\mathbb Q=0.
\]
Thus \(C(U)\) is finite for every subgroup \(U\leq G\). Since
\[
F(U)=H^1(U,M)
\]
is also finite, both \(C\) and \(F\) decompose functorially into their primary
parts:
\[
C\cong \bigoplus_p C\otimes_{\mathbb Z}\mathbb Z_p,
\qquad
F\cong \bigoplus_p F\otimes_{\mathbb Z}\mathbb Z_p.
\]
The \(p\)-local isomorphisms above therefore glue to a natural isomorphism
\[
C\cong F
\]
of cohomological Mackey functors. Composing the quotient map
\[
FP_{Q_0}\longrightarrow C
\]
with this isomorphism gives the desired exact sequence
\[
FP_{Q_1}
\xrightarrow{\;FP_{\Psi}\;}
FP_{Q_0}
\longrightarrow
F
\longrightarrow 0.
\]
\end{proof}

\section{The \(K,L\)-Brauer presentation}\label{brauerparametrization}

The construction above is \(G\)-equivariant. Thus, if the image of the Galois
action on the fan is a subgroup \(H\leq G\), the same method applies after
restricting the relevant \(G\)-lattices and \(G\)-sets to \(H\). In that
generality, the fields \(K\), \(L\), and \(KL\) below are replaced by the finite
étale \(k\)-algebras attached to the corresponding restricted \(H\)-sets. To
keep the notation concrete, we state the Brauer parametrization in the case
where the chosen Galois splitting field \(E/k\) has
\(\operatorname{Gal}(E/k)\simeq G\) and these finite étale \(k\)-algebras are
field extensions.

We now translate the permutation-projective presentation of \(F=H^1(-,M)\) into
a Brauer-theoretic description of \(H^1(k,T)\). We apply
\cite[Theorem~6.1]{retforms} to the exact sequence
\[
FP_{Q_1}
\xrightarrow{\,FP_{\Psi}\,}
FP_{Q_0}
\longrightarrow
F
\longrightarrow 0
\]
from Theorem~\ref{thm:integral-first-differential-H0}, where
\[
Q_1=
\mathbb Z[S_{n+1}/S_n]\otimes_{\mathbb Z}\mathbb Z[S_2]
\oplus \mathbb Z^{\oplus 2}
\]
and
\[
Q_0=
\mathbb Z[S_2]\oplus \mathbb Z[S_{n+1}/S_n].
\]

For a finite Galois extension \(E/k\) with
\[
\operatorname{Gal}(E/k)\cong G=S_{n+1}\times S_2,
\]
set
\[
H_K:=S_{n+1}\times 1,
\qquad
H_L:=S_n\times S_2,
\qquad
H_{KL}:=S_n\times 1,
\]
and
\[
K:=E^{H_K},
\qquad
L:=E^{H_L},
\qquad
KL:=E^{H_{KL}}.
\]
We write
\[
\Br(E/F):=\ker\!\bigl(\Br(F)\to \Br(E)\bigr)
\]
for the relative Brauer group.

\begin{theorem}\label{thm:gen-dPn-relative-brauer}
Let \(n\ge 2\) be even. Let \(X\) be a \(k\)-form of the generalized del Pezzo
variety \(V_n\), and let \(T\) be the dense \(k\)-torus acting faithfully on \(X\)
with an open orbit. Assume that \(E/k\) is the minimal Galois splitting field of
\(T\), and that the chosen isomorphism
\[
\operatorname{Gal}(E/k)\cong S_{n+1}\times S_2
\]
realizes the specified \(G\)-lattice structure on \(M\).

Then \(K/k\) is quadratic, \(L/k\) has degree \(n+1\), and
\[
KL\simeq K\otimes_k L.
\]

Moreover, \(H^1(k,T)\) is naturally identified with the subgroup of pairs
\[
(B,Q)\in \Br(K)\times \Br(L)
\]
satisfying
\[
\operatorname{Res}_{E/K}(B)=0,
\qquad
\operatorname{Res}_{E/L}(Q)=0,
\]
\[
\operatorname{Cor}_{K/k}(B)=0,
\qquad
\operatorname{Cor}_{L/k}(Q)=0,
\]
and
\[
\operatorname{Res}_{KL/K}(B)+\operatorname{Res}_{KL/L}(Q)=0
\quad\text{in }\Br(KL).
\]
For every such pair, one has
\[
(n+1)B=0,
\qquad
2Q=0.
\]
\end{theorem}

\begin{proof}
First, by the Galois correspondence,
\[
[K:k]=[G:H_K]=2,
\qquad
[L:k]=[G:H_L]=n+1.
\]
Also
\[
H_K\cap H_L=H_{KL},
\qquad
\langle H_K,H_L\rangle=G.
\]
Hence
\[
E^{H_K}E^{H_L}=E^{H_K\cap H_L}=E^{H_{KL}}
\]
and
\[
E^{H_K}\cap E^{H_L}=E^G=k.
\]
Therefore
\[
KL=E^{H_{KL}}\simeq K\otimes_k L.
\]

Now apply \cite[Theorem~6.1]{retforms} to the permutation-projective
presentation
\[
FP_{Q_1}
\xrightarrow{\,FP_{\Psi}\,}
FP_{Q_0}
\longrightarrow
F
\longrightarrow 0.
\]
The summands of \(Q_0\) are
\[
\mathbb Z[S_2]\cong \mathbb Z[G/H_K],
\qquad
\mathbb Z[S_{n+1}/S_n]\cong \mathbb Z[G/H_L],
\]
so they give the source Brauer groups
\[
\Br(E/K)\oplus \Br(E/L).
\]
The summands of \(Q_1\) are
\[
\mathbb Z[S_{n+1}/S_n]\otimes_{\mathbb Z}\mathbb Z[S_2]
\cong \mathbb Z[G/H_{KL}]
\]
and two copies of \(\mathbb Z[G/G]\). Hence they give the target groups
\[
\Br(E/KL)\oplus \Br(E/k)^{\oplus 2}.
\]
Under this identification, the map induced by \(\Psi\) is
\[
(B,Q)\longmapsto
\Bigl(
\operatorname{Res}_{KL/K}(B)+\operatorname{Res}_{KL/L}(Q),\,
\operatorname{Cor}_{K/k}(B),\,
\operatorname{Cor}_{L/k}(Q)
\Bigr).
\]
Thus
\[
H^1(k,T)\cong
\ker\!\left(
\Br(E/K)\oplus \Br(E/L)
\longrightarrow
\Br(E/KL)\oplus \Br(E/k)^{\oplus 2}
\right).
\]
Since
\[
\Br(E/K)=\ker\bigl(\Br(K)\to\Br(E)\bigr),
\qquad
\Br(E/L)=\ker\bigl(\Br(L)\to\Br(E)\bigr),
\]
this is exactly the stated description in terms of pairs
\[
(B,Q)\in \Br(K)\times \Br(L).
\]

It remains to prove the torsion statements. Since
\[
[KL:K]=n+1,
\qquad
[KL:L]=2,
\]
the relation
\[
\operatorname{Res}_{KL/K}(B)+\operatorname{Res}_{KL/L}(Q)=0
\]
gives, after applying \(\operatorname{Cor}_{KL/K}\),
\[
(n+1)B
=
-\operatorname{Cor}_{KL/K}\operatorname{Res}_{KL/L}(Q).
\]
Because \(KL\simeq K\otimes_k L\), restriction-corestriction compatibility gives
\[
\operatorname{Cor}_{KL/K}\operatorname{Res}_{KL/L}(Q)
=
\operatorname{Res}_{K/k}\operatorname{Cor}_{L/k}(Q)=0.
\]
Hence
\[
(n+1)B=0.
\]
Similarly, applying \(\operatorname{Cor}_{KL/L}\) gives
\[
2Q
=
-\operatorname{Cor}_{KL/L}\operatorname{Res}_{KL/K}(B)
=
-\operatorname{Res}_{L/k}\operatorname{Cor}_{K/k}(B)=0.
\]
Thus
\[
(n+1)B=0,
\qquad
2Q=0.
\]
\end{proof}

\section{Degree-zero source fields}

\subsection{The distinguished degree-zero summands}

In this section we compare the distinguished degree-zero permutation term with
other possible degree-zero terms in permutation-projective presentations of
\[
F=H^1(-,M).
\]
We work after \(p\)-adic completion. At \(p=2\), the relevant summand is the
\(\mathbb Z_2G\)-lattice \(A_2\). For each odd prime \(p\mid(n+1)\), the
relevant summand is the \(\mathbb Z_pG\)-lattice \(\mathbb Z_p^-\). We determine
when these occur as direct summands of completed permutation lattices.

Let
\[
\Omega:=\{0,\dots,n\}.
\]
Recall that \(S_{n-1}\le S_{n+1}\) is the pointwise stabilizer of two points.
Since \(n+1\) is odd, the integer \(n-1\) is also odd. We choose
\[
P_1\in \operatorname{Syl}_2(S_{n-1})
\]
so that \(P_1\) fixes exactly the three points
\[
\{n-2,n-1,n\}.
\]

We use the Brauer--Green correspondence for \(p\)-permutation modules: an
indecomposable trivial-source summand with vertex \(P\) is detected by the
corresponding indecomposable projective summand of the Brauer quotient at \(P\);
see, for example, \cite{Broue1985ScottMP,Lassueur2023}.

\begin{lemma}[Vertices of the degree-zero summands]
\label{lem:vertices-degree-zero-labels}
The following hold.

\begin{enumerate}[label=\textup{(\roman*)}]
\item As a \(\mathbb Z_2G\)-module, \(A_2\) has vertex
\[
P_1\times S_2.
\]
Moreover, after reduction modulo \(2\), its Green correspondent is inflated from
the \(2\)-dimensional simple \(\mathbb F_2S_3\)-module \(D\), where
\[
S_3=\operatorname{Sym}\{n-2,n-1,n\}.
\]

\item For every odd prime \(p\mid(n+1)\), the \(\mathbb Z_pG\)-module
\(\mathbb Z_p^-\) has vertex a Sylow \(p\)-subgroup of \(G\).
\end{enumerate}
\end{lemma}

\begin{proof}
For a \(p\)-subgroup \(Q\leq G\), recall that the Brauer quotient of a
\(\mathbb Z_pG\)-lattice \(U\) at \(Q\) is
\[
U(Q):=
U^Q\Big/
\left(
\sum_{R<Q}\operatorname{Tr}_R^Q(U^R)+pU^Q
\right),
\]
viewed as an \(\mathbb F_p[N_G(Q)/Q]\)-module. We use the Brauer-quotient
criterion for vertices of \(p\)-permutation modules: if \(U\) is indecomposable,
then its vertices are the maximal \(p\)-subgroups \(Q\) such that \(U(Q)\neq 0\)
\cite[Proposition~4.12]{Lassueur2023}.

For \textup{(i)}, let
\[
\overline A_2=\mathbb F_2\otimes_{\mathbb Z_2}A_2.
\]
By definition, \(\overline A_2\) is the kernel of the sum-of-coefficients map
\[
\mathbb F_2[\Omega]\longrightarrow\mathbb F_2,
\]
and the \(S_2\)-factor acts trivially. Let \(Q\leq G\) be a \(2\)-subgroup, and
let \(Q_0\leq S_{n+1}\) be its projection. Since \(n+1\) is odd, \(Q_0\) has a
fixed point on \(\Omega\). In the Brauer quotient, every non-fixed
\(Q_0\)-orbit is killed by transfer: if \(O\) is such an orbit, \(x\in O\), and
\(y\in\Omega^{Q_0}\), then the orbit sum over \(O\) is the transfer of
\([x]+[y]\), because \(|O|\) is even in characteristic \(2\). Hence only the
fixed points survive, and \(\overline A_2(Q)\) is naturally the kernel of
\[
\mathbb F_2[\Omega^{Q_0}]\longrightarrow\mathbb F_2.
\]
This kernel is nonzero exactly when \(Q_0\) fixes at least two points.

Thus the maximal \(2\)-subgroups with nonzero Brauer quotient are obtained by
taking a Sylow \(2\)-subgroup of the pointwise stabilizer of two points in
\(S_{n+1}\), together with the central \(S_2\)-factor. These are conjugate to
\[
P_1\times S_2.
\]
Therefore \(A_2\) has vertex \(P_1\times S_2\).

By the choice of \(P_1\), its fixed-point set is
\[
\{n-2,n-1,n\}.
\]
At the vertex \(P_1\times S_2\), the Brauer quotient is the kernel of
\[
\mathbb F_2[\{n-2,n-1,n\}]\longrightarrow\mathbb F_2.
\]
The normalizer quotient acts on these three points through
\[
S_3=\operatorname{Sym}\{n-2,n-1,n\},
\]
and this kernel is the \(2\)-dimensional simple \(\mathbb F_2S_3\)-module
\(D\). Hence the Green correspondent is inflated from \(D\)
\cite[Theorem~4.6]{Lassueur2023}.

For \textup{(ii)}, let \(p\mid(n+1)\) be odd. The sign character of \(S_2\) is
trivial on \(p\)-subgroups, so \(\mathbb Z_p^-\) has nonzero Brauer quotient at
every \(p\)-subgroup of \(G\). The maximal such subgroups are the Sylow
\(p\)-subgroups. Since \(S_2\) has no \(p\)-part, such a vertex has the form
\[
P\times 1,
\]
with \(P\in\operatorname{Syl}_p(S_{n+1})\).
\end{proof}

\subsection{Admissible stabilizers and local summand criteria}

Recall that \(P_1\le S_{n+1}\) is a Sylow \(2\)-subgroup of the pointwise
stabilizer of two points, chosen so that
\[
\Omega^{P_1}=\{n-2,n-1,n\}.
\]

\begin{definition}[Admissible stabilizers]
\label{def:admissible-stabilizers}
Let \(H\le G\).

\begin{enumerate}[label=\textup{(\roman*)}]
\item We say that \(H\) is \emph{\(2\)-admissible} if \(H=J\times S_2\) for some \(J\le S_{n+1}\), and there exists \(g\in S_{n+1}\) such that, writing \(P:={}^{g}P_1,\)
one has \(P\le J\) and \(N_J(P)\) fixes a point of \(\Omega^P\).

\item Let \(p\mid(n+1)\) be odd. We say that \(H\) is
\emph{\(p\)-admissible} if \(H\le S_{n+1}\times 1\) and \(H\) contains a Sylow \(p\)-subgroup of \(G\).
\end{enumerate}
\end{definition}

\begin{prop}[The \(2\)-primary direct-summand criterion]
\label{prop:two-primary-local-criterion}
Let \(H\le G\). Then
\[
A_2\mid \mathbb Z_2[G/H]
\]
if and only if \(H\) is \(2\)-admissible.
\end{prop}

\begin{proof}
Suppose first that
\[
A_2\mid \mathbb Z_2[G/H].
\]
By Lemma~\ref{lem:vertices-degree-zero-labels}, \(A_2\) has vertex
\[
P_1\times S_2.
\]
Hence \(P_1\times S_2\) is contained in a conjugate of \(H\). Since
\(1\times S_2\) is central in \(G\), this implies
\[
1\times S_2\leq H.
\]
Thus
\[
H=J\times S_2,
\qquad
J=\operatorname{pr}_1(H)\leq S_{n+1}.
\]

For such \(H\), the \(G\)-set \(G/H\) is the inflation of the
\(S_{n+1}\)-set \(S_{n+1}/J\). Thus
\[
\mathbb Z_2[G/H]
\cong
\operatorname{Inf}_{S_{n+1}}^G
\mathbb Z_2[S_{n+1}/J].
\]
By idempotent lifting, it is enough to test the corresponding summand after
reduction modulo \(2\).

Let
\[
S=S_{n+1},
\qquad
P=P_1,
\qquad
N=N_S(P),
\qquad
\overline N=N/P.
\]
For a permutation module \(\mathbb F_2[X]\), the Brauer quotient at \(P\) is
\[
\mathbb F_2[X^P],
\]
as a module for \(N/P\): the non-fixed \(P\)-orbits are killed by relative
transfers, while the \(P\)-fixed points survive. Applying this to \(X=S/J\), a
coset \(gJ\) is \(P\)-fixed exactly when
\[
P\leq {}^gJ.
\]
The group \(N\) acts on these fixed cosets, and the stabilizer of \(gJ\) in
\(N\) is
\[
N\cap {}^gJ.
\]
Therefore \(\mathbb F_2[S/J](P)\) is the direct sum, over representatives
\(gJ\) of the \(N\)-orbits in \((S/J)^P\), of the permutation modules
\[
\mathbb F_2[\overline N/\overline J_g],
\]
where
\[
\overline J_g=(N\cap {}^gJ)/P.
\]

By the Brauer-Green correspondence for \(p\)-permutation modules
\cite[Theorem~4.6]{Lassueur2023}, the summand with vertex \(P_1\) and Green
correspondent \(D\) occurs precisely when \(D\) occurs in the Brauer quotient at
\(P_1\). Let \(R_g\) be the image of \(\overline J_g\) in
\[
\operatorname{Sym}(\Omega^P)
=
\operatorname{Sym}\{n-2,n-1,n\}
\cong S_3.
\]
Since
\[
\mathbb F_2[\overline N/\overline J_g]
\cong
\operatorname{Ind}_{\overline J_g}^{\overline N}\mathbb F_2,
\]
Frobenius reciprocity gives
\[
D\mid \mathbb F_2[\overline N/\overline J_g]
\]
if and only if
\[
D^{\overline J_g}\neq 0,
\]
equivalently
\[
D^{R_g}\neq 0.
\]
Here \(D\) is inflated from the \(2\)-dimensional simple
\(\mathbb F_2S_3\)-module, namely the kernel of
\[
\mathbb F_2[\Omega^P]\to\mathbb F_2.
\]
Its \(R_g\)-fixed space is nonzero exactly when \(R_g\) fixes a point of
\(\Omega^P\): if \(R_g\) fixes a point, there is a nonzero fixed vector of
coefficient sum zero, while if \(R_g\) is transitive, the only fixed vectors in
\(\mathbb F_2[\Omega^P]\) are multiples of the full orbit sum, which does not
have coefficient sum zero.

Conjugating back, this is exactly the condition that, for
\[
P':={}^{g^{-1}}P_1\leq J,
\]
the subgroup \(N_J(P')\) fixes a point of \(\Omega^{P'}\). Hence \(H\) is
\(2\)-admissible.

Conversely, suppose \(H\) is \(2\)-admissible. Then
\[
H=J\times S_2
\]
for some \(J\leq S_{n+1}\), and there exists a conjugate
\[
P'={}^{g^{-1}}P_1\leq J
\]
such that \(N_J(P')\) fixes a point of \(\Omega^{P'}\). Equivalently, after
conjugating by \(g\), the image of
\[
(N\cap {}^gJ)/P_1
\]
in
\[
\operatorname{Sym}(\Omega^{P_1})\cong S_3
\]
fixes a point. Hence \(D\) occurs in the Brauer quotient of
\[
\mathbb F_2[S_{n+1}/J]
\]
at \(P_1\). By the Brauer-Green correspondence and idempotent lifting, the
corresponding \(\mathbb Z_2G\)-summand is \(A_2\). Therefore
\[
A_2\mid \mathbb Z_2[G/H].
\]
\end{proof}

\begin{prop}[Odd-primary direct-summand criterion]
\label{prop:odd-primary-local-criterion}
Let \(p\mid(n+1)\) be odd, and let \(H\le G\). Then
\[
\mathbb Z_p^-\mid \mathbb Z_p[G/H]
\]
if and only if \(H\) is \(p\)-admissible.
\end{prop}

\begin{proof}
Suppose first that
\[
\mathbb Z_p^-\mid \mathbb Z_p[G/H].
\]
Then
\[
\operatorname{Hom}_{\mathbb Z_pG}
(\mathbb Z_p^-,\mathbb Z_p[G/H])
\neq 0.
\]
By Frobenius reciprocity,
\[
\operatorname{Hom}_{\mathbb Z_pG}
(\mathbb Z_p^-,\mathbb Z_p[G/H])
\cong
\operatorname{Hom}_{\mathbb Z_pH}(\mathbb Z_p^-,\mathbb Z_p).
\]
The right-hand side is nonzero only if the sign character is trivial on \(H\).
Hence
\[
H\leq S_{n+1}\times 1.
\]
Write
\[
H=H_0\times 1,
\qquad
H_0\leq S_{n+1}.
\]
By Lemma~\ref{lem:vertices-degree-zero-labels}, \(\mathbb Z_p^-\) has vertex a
Sylow \(p\)-subgroup of \(G\), namely one of the form
\[
P\times 1,
\qquad
P\in \operatorname{Syl}_p(S_{n+1}).
\]
Since \(\mathbb Z_p^-\) is a direct summand of \(\mathbb Z_p[G/H]\), this
vertex is contained in a conjugate of \(H\). Therefore \(H_0\) contains a Sylow
\(p\)-subgroup of \(S_{n+1}\), equivalently
\[
p\nmid [S_{n+1}:H_0].
\]
Thus \(H\) is \(p\)-admissible.

Conversely, suppose \(H\) is \(p\)-admissible. Then
\[
H=H_0\times 1
\]
for some \(H_0\leq S_{n+1}\), and
\[
p\nmid [S_{n+1}:H_0].
\]
Since \(p\) is odd, this is equivalent to
\[
p\nmid [G:H].
\]
The sign character is trivial on \(H\), so define maps
\[
\mathbb Z_p^- \longrightarrow \mathbb Z_p[G/H],
\qquad
1\longmapsto \sum_{gH\in G/H} \operatorname{sgn}(g)\, gH,
\]
and
\[
\mathbb Z_p[G/H]\longrightarrow \mathbb Z_p^-,
\qquad
gH\longmapsto \operatorname{sgn}(g).
\]
Their composition is multiplication by \([G:H]\), which is a unit in
\(\mathbb Z_p\). Hence \(\mathbb Z_p^-\) splits off from
\(\mathbb Z_p[G/H]\). Therefore
\[
\mathbb Z_p^-\mid \mathbb Z_p[G/H].
\]
\end{proof}

\subsection{Classification of degree-zero source terms}

Let
\[
X_0=\bigsqcup_{i\in I}G/H_i
\]
be a finite permutation \(G\)-set, and write
\[
H^0(X_0):=H^0(-,\mathbb Z[X_0]).
\]

\begin{theorem}[Degree-zero permutation terms]
\label{thm:degree-zero-local-classification}
Let
\[
X_0=\bigsqcup_{i\in I}G/H_i
\]
be a finite permutation \(G\)-set. The following are equivalent:

\begin{enumerate}[label=\textup{(\alph*)}]
\item The cohomological Mackey functor
\[
H^0(-,\mathbb Z[X_0])
\]
can occur as the degree-zero term of a permutation presentation of
\[
H^1(-,M).
\]

\item The stabilizers \(H_i\) satisfy the following admissibility conditions:
\begin{enumerate}[label=\textup{(\roman*)}]
\item at least one \(H_i\) is \(2\)-admissible;

\item for every odd prime \(p\mid(n+1)\), at least one \(H_i\) is
\(p\)-admissible.
\end{enumerate}
\end{enumerate}
\end{theorem}

\begin{proof}
The completed minimal projective presentation of \(H^1(-,M)\) has, at the prime
\(2\), degree-zero summand \(A_2\). For each odd prime \(p\mid(n+1)\), it has
degree-zero summand \(\mathbb Z_p^-\). By Proposition~7.7 of \cite{retforms}, a
permutation degree-zero term can occur precisely when, after completion at each
prime, it contains all indecomposable summands appearing in the completed
minimal degree-zero term.

Since
\[
X_0=\bigsqcup_{i\in I}G/H_i,
\]
the completed permutation lattice decomposes as the direct sum of the completed
lattices attached to the individual orbits \(G/H_i\). Thus the above condition
is equivalent to requiring at least one orbit \(G/H_i\) with
\[
A_2\mid \mathbb Z_2[G/H_i],
\]
and, for every odd prime \(p\mid(n+1)\), at least one orbit \(G/H_i\) with
\[
\mathbb Z_p^-\mid \mathbb Z_p[G/H_i].
\]
By Proposition~\ref{prop:two-primary-local-criterion}, the first condition is
equivalent to \(H_i\) being \(2\)-admissible. By
Proposition~\ref{prop:odd-primary-local-criterion}, the second condition is
equivalent to \(H_i\) being \(p\)-admissible. This proves the equivalence of
\textup{(a)} and \textup{(b)}.
\end{proof}

We now compare the field-theoretic admissibility conditions from
Definition~\ref{def:admissible-source-fields} with the stabilizer conditions of
Definition~\ref{def:admissible-stabilizers}.

In the full splitting case, put
\[
\widetilde L:=E^{1\times S_2}.
\]
Thus \(\widetilde L/k\) is the Galois closure of \(L/k\) inside \(E\), with
Galois group \(S_{n+1}\). Choose a primitive element \(L=k(\alpha_0)\), and let
\(\alpha_0,\dots,\alpha_n\) be its \(S_{n+1}\)-conjugates in \(\widetilde L\).
Write
\[
L_i:=k(\alpha_i),
\qquad
\Omega:=\{0,\dots,n\}.
\]
Thus the \(k\)-conjugates of \(L\) inside \(\widetilde L\) are precisely
\(L_0,\dots,L_n\).

\begin{lemma}\label{lem:field-subgroup-admissible}
Let \(F_0=E^H\) for a subgroup \(H\le G\). Then:

\begin{enumerate}[label=\textup{(\roman*)}]
\item \(F_0\) is \(2\)-admissible if and only if \(H\) is \(2\)-admissible.

\item For every odd prime \(p\mid(n+1)\), \(F_0\) is \(p\)-admissible if and
only if \(H\) is \(p\)-admissible.
\end{enumerate}
\end{lemma}

\begin{proof}
We first prove \textup{(i)}. By the Galois correspondence,
\[
F_0\subseteq \widetilde L
\quad\Longleftrightarrow\quad
1\times S_2\le H.
\]
In this case,
\[
H=J\times S_2,
\qquad
J=\operatorname{Gal}(\widetilde L/F_0)\le S_{n+1}.
\]

Assume first that \(H\) is \(2\)-admissible. Then \(J\) satisfies the
\(2\)-admissibility criterion. Hence there exist distinct indices \(i,j\in\Omega\)
and
\[
P\in \operatorname{Syl}_2(S_{\Omega\setminus\{i,j\}})
\]
such that
\[
P\le J
\]
and \(N_J(P)\) fixes a point of \(\Omega^P\). Let \(t\in\Omega^P\) be such a fixed
point. Since \(P\) is conjugate to \(P_1\), the fixed-point set \(\Omega^P\) has
three elements. Hence we may choose \(s\in\Omega^P\), \(s\neq t\).

Since \(P\) fixes both \(t\) and \(s\), we have
\[
P\le S_{\Omega\setminus\{t,s\}}.
\]
Moreover, \(P\) has the order of a Sylow \(2\)-subgroup of
\(S_{\Omega\setminus\{i,j\}}\), and
\[
|S_{\Omega\setminus\{i,j\}}|
=
|S_{\Omega\setminus\{t,s\}}|.
\]
Therefore
\[
P\in \operatorname{Syl}_2(S_{\Omega\setminus\{t,s\}}).
\]

Set
\[
F^\sharp:=\widetilde L^P.
\]
Since
\[
L_tL_s=\widetilde L^{S_{\Omega\setminus\{t,s\}}},
\]
the field \(F^\sharp\) is a maximal odd-degree subextension of
\[
\widetilde L/L_tL_s.
\]
Also, since
\[
P\le J\cap S_{\Omega\setminus\{t,s\}},
\]
we have
\[
F_0L_tL_s
=
\widetilde L^{J\cap S_{\Omega\setminus\{t,s\}}}
\subseteq
\widetilde L^P
=
F^\sharp.
\]

Let \(R\) be the smallest intermediate field
\[
F_0\subseteq R\subseteq F^\sharp
\]
such that \(F^\sharp/R\) is Galois. As before, this field is
\[
R=\widetilde L^{N_J(P)}.
\]
Since \(N_J(P)\) fixes \(t\), we have
\[
N_J(P)\le S_{\Omega\setminus\{t\}},
\]
and hence
\[
L_t=\widetilde L^{S_{\Omega\setminus\{t\}}}\subseteq
\widetilde L^{N_J(P)}=R.
\]
Thus \(F_0\) is \(2\)-admissible in the sense of
Definition~\ref{def:admissible-source-fields}, with
\[
L'=L_t,
\qquad
L''=L_s.
\]

Conversely, assume that \(F_0\) is \(2\)-admissible in the sense of
Definition~\ref{def:admissible-source-fields}. Then \(F_0\subseteq\widetilde L\),
so
\[
H=J\times S_2,
\qquad
J=\operatorname{Gal}(\widetilde L/F_0).
\]
Choose the two conjugates in the definition as
\[
L'=L_i,
\qquad
L''=L_j.
\]
Let \(F^\sharp\) be the corresponding maximal odd-degree subextension. Since
\(F^\sharp\) is a maximal odd-degree subextension of
\[
\widetilde L/L_iL_j,
\]
there is a Sylow \(2\)-subgroup
\[
P\in\operatorname{Syl}_2(S_{\Omega\setminus\{i,j\}})
\]
such that
\[
F^\sharp=\widetilde L^P.
\]
The inclusion
\[
F_0L_iL_j\subseteq F^\sharp
\]
implies
\[
P\le J\cap S_{\Omega\setminus\{i,j\}},
\]
and hence \(P\le J\).

As above, the smallest field \(R\) with
\[
F_0\subseteq R\subseteq F^\sharp
\]
and \(F^\sharp/R\) Galois is
\[
R=\widetilde L^{N_J(P)}.
\]
By \(2\)-admissibility, Definition~\ref{def:admissible-source-fields} gives
\[
L_i=L'\subseteq R.
\]
Since
\[
L_i=\widetilde L^{S_{\Omega\setminus\{i\}}},
\]
this implies
\[
N_J(P)\le S_{\Omega\setminus\{i\}}.
\]
Equivalently, \(N_J(P)\) fixes the point \(i\). Since \(i\in\Omega^P\), the image
of \(N_J(P)\) in
\[
\operatorname{Sym}(\Omega^P)\cong S_3
\]
has a fixed point. Hence \(J\) satisfies the \(2\)-admissibility criterion, and
therefore \(H=J\times S_2\) is \(2\)-admissible.

Now let \(p\mid(n+1)\) be odd. By the Galois correspondence,
\[
K\subseteq F_0
\quad\Longleftrightarrow\quad
H\le S_{n+1}\times 1.
\]
Writing \(H=J\times 1\) with \(J\le S_{n+1}\), one has
\[
[F_0:K]=[S_{n+1}:J].
\]
Thus
\[
p\nmid [F_0:K]
\]
if and only if \(J\) contains a Sylow \(p\)-subgroup of \(S_{n+1}\), equivalently
if and only if \(H\) contains a \(G\)-conjugate of \(P_p\times 1\). This is
precisely \(p\)-admissibility for \(H\).
\end{proof}

\begin{corollary}[Field-theoretic form of the degree-zero criterion]
\label{cor:field-theoretic-degree-zero}
Let \(F_1,\dots,F_r\) be intermediate fields of \(E/k\), and write
\(F_i=E^{H_i}\). There exists a relative Brauer presentation of \(H^1(k,T)\),
arising from a permutation-projective presentation of \(H^1(-,M)\), with source
fields \(F_1,\dots,F_r\), if and only if the following conditions hold:
\begin{enumerate}[label=\textup{(\roman*)}]
\item at least one \(F_i\) is \(2\)-admissible;

\item for every odd prime \(p\mid(n+1)\), at least one \(F_i\) is
\(p\)-admissible.
\end{enumerate}
\end{corollary}

\begin{proof}
By Theorem~\ref{thm:degree-zero-local-classification}, the degree-zero term can
occur precisely when the stabilizers \(H_i\) include at least one
\(2\)-admissible subgroup and, for every odd prime \(p\mid(n+1)\), at least one
\(p\)-admissible subgroup. By Lemma~\ref{lem:field-subgroup-admissible}, this is
equivalent to the stated admissibility conditions on the fields \(F_i=E^{H_i}\).
\end{proof}

\begin{remark}
Corollary~\ref{cor:field-theoretic-degree-zero} classifies the possible source
fields in the degree-zero Brauer term. It does not determine the target fields,
or equivalently the restriction and corestriction relations; these depend on the
chosen degree-one permutation term and the differential of the presentation.
\end{remark}

\medskip
\bibliographystyle{alpha}

\end{document}